\documentclass[12pt,bezier]{article}
\usepackage{times}
\usepackage{booktabs}
\usepackage{pifont}
\usepackage{floatrow}
\usepackage{caption}
\usepackage{mathrsfs}
\usepackage[fleqn]{amsmath}
\usepackage{amsfonts,amsthm,amssymb,mathrsfs,bbding}
\usepackage{txfonts}
\usepackage{graphics,multicol}
\usepackage{graphicx}
\usepackage{color}
\usepackage{caption}
\usepackage{cite}
\usepackage{latexsym,bm}
\usepackage{indentfirst}

\evensidemargin=\oddsidemargin\topmargin=-1.5cm

\newtheorem{lem}{Lemma}[section]
\newtheorem{thm}{Theorem}[section]

\theoremstyle{definition}
\renewcommand\proofname{\bf{Proof}}
\addtocounter{section}{0}

\begin{document}
	
\title{Maxima of the $Q$-index: Graphs with no $K_{1,t}$-minor\footnote{This work is supported
by National Natural Science Foundation of China (No. 12061074),
the China Postdoctoral Science Foundation (No. 2019M661398). }}
\author{
{Yanting Zhang$^{a}$, Zhenzhen Lou$^{a,b}$\footnote{Corresponding author.
\it Email addresses:
 xjdxlzz@163.com (Z. Lou).} }\\[2mm]
\small $^a$ School of Mathematics, East China University of Science and Technology\\
\small Shanghai 200237, China\\
\small $^b$ College of Mathematics and Systems Science, Xinjiang University,Urumqi, \\
\small  Xinjiang, 830046, China
}
\date{}
\maketitle{\flushleft\large\bf Abstract}
A graph is said to be \textit{$H$-minor free} if it does not contain $H$ as a minor.
In this paper, we characteristic the unique extremal graph with maximal $Q$-index among all $n$-vertex $K_{1,t}$-minor free graphs ($t\ge3$).

\begin{flushleft}
\end{flushleft}

\section{Introduction}
Let $G$ be an undirected simple graph with vertex set $V(G)=\{v_{1}, \ldots, v_{n}\}$ and edge set $E(G)$.
For $u \in V(G)$, set $N_{G}(u)=\{v \mid uv \in E(G) \}$,  $d_{G}(u)=|N_{G}(u)|$, (or $N(u)$, $d(u)$ for short, respectively), and $\overline{N}_{G}(u)=\{v \mid uv \notin E(G) \}$. Let $\delta(G)$ and $\Delta(G)$ denote the minimum degree and maximum degree of graph $G$, respectively.
The \textit{adjacency matrix} $A(G)$ of $G$ is the $n \times n$ matrix $(a_{ij})$, where $a_{ij}=1$
if $v_{i}$ is adjacent to $v_{j}$, and $0$ otherwise.
The \textit{$Q$-matrix} (or \textit{signless Laplacian matrix}) of $G$ is defined as $Q(G)=D(G)+A(G)$,
where $D(G)$ is the diagonal matrix of vertex degree of $G$. The largest eigenvalue of $Q(G)$,
denoted by $q_{1}(G)$, is called the \textit{$Q$-index} (or \textit{signless Laplacian spectral radius}) of $G$.
As usual, denote by $K_{n}$, $K_{s,t}$ the \textit{complete graph}  and \textit{complete bipartite graph} of order $n$, respectively. We denote  $ S^{n-t} (K_{t})$ a graph obtained from $K_{t}$ by subdividing $n-t$ times of one edge.
Clearly,  $S^{n-t}(K_{t})$ is a $K_{1,t}$-minor free graph of order $n \geq t+1$ and $t\geq3$.
Denote by $F_{s, t}(n):=K_{s-1} \vee (p K_{t} \cup K_{r})$, where $n-s+1=pt+r$ and $0 \leq r<t$.
Given two graphs $H$ and $G$, $H$ is a \textit{minor} of $G$
if $H$ can be obtained from a subgraph of $G$ by contracting edges.
A graph is said to be \textit{$H$-minor free} if it does not contain $H$ as a minor.

Minors play a key role in graph theory, and extremal problems on forbidding
minors have attracted appreciable amount of interests. 
It well know that every planar graph is $\{K_{3,3}, K_5\}$-minor free and every outer-planar graph is $\{K_{2,3}, K_4\}$-minor free. 
The classical result about minors one can see \cite{B.N,D.C-0,M.N} for example. 
In this paper we concentre our attention on  $K_{s,t}$-minor free graphs.
In 2017, Nikiforov \cite{Nikiforov} 
established a sharp upper bound of adjacency  spectral radius over all $K_{2,t}$-minor free graphs of large
order for $t\ge4$, and determined all extremal graphs. In 2019,
Tait \cite{M.T} 
obtained an upper bound of   adjacency spectral radius 
over all  $K_{s,t}$-minor free graphs of large order for  $2\leq s\leq t$, and determined  the   extremal graph when 
$t \mid  n-s+1$.
Very recently, Zhai and Lin \cite{M.Q} improved  Tait's result by removing the condition $t \mid n-s+1$.
Meanwhile, they
determined the extremal graph over all the $K_{1,t}$-minor free graphs.
Thus the adjacency spectral extremal problem on $K_{s,t}$-minor free graphs
is completely solved for large enough $n$.

In spectral extremal graph theory, determine the maximum $Q$-index of graphs attract much attentions.
So we want to know what is the maximum $Q$-index among all $K_{s,t}$-minor free graphs.
In 2021, Chen and Zhang \cite {M.C}
proved $F_{2,t}(n)$ has the maximum  $Q$-index among all $K_{2,t}$-minor free graphs, where $t\ge 3$.
Additionally,
for $t\ge s\ge2$, Chen and Zhang \cite {M.C} conjectured that $F_{s,t}(n)$ has the maximum $Q$-index among all $K_{s,t}$-minor free graph of sufficiently large order $n$. Thus the problem of the maximum $Q$-index among all $K_{1,t}$-minor free graphs is still open.
In this paper, we determine the unique extremal graph with the maximum $Q$-index among all $n$-vertices $K_{1,t}$-minor free graphs as follows.

\begin{thm}\label{thm::1.5}
Let $t \geq 3$, and $G^{*}$ be a $K_{1,t}$-minor free graph of order $n =t+1 $ with the maximum $Q$-index. Then		
\begin{eqnarray*}
G^{*} \cong
\begin{cases}
K_{n} \backslash \lfloor\frac{n}{2}\rfloor K_{1,1} &\text{ for $n =t+1$,} \\
S^{n-t}(K_{t}) &\text{for $n \geq t+2 $.}
\end{cases}
\end{eqnarray*}			
\end{thm}

\section{Preliminaries}
In this section, we will list some symbols and some useful lemmas.
For $A, B \subseteq V(G)$, $e(A)$
is denoted the number of the edges of $G$ with both end vertices in $A$ and $e(A, B)$ is denoted the number of the edges of $G$ with one end vertex in $A$ and the other in $B$.
For two vertex disjoint graphs $G$ and $H$, we denote by $G \cup H$ and $G \vee H$ the union of $G$ and $H$, and the join of $G$ and $H$, i.e., joining every vertex of $G$ to every vertex of $H$.
Denote by $kG$ the union of $k$ disjoint copies of $G$.
For $A \subseteq V(G)$, the graph $G[A]$ is the \textit{induced subgraph} by $A$.
Let $G-e$ be the graph obtained by deleting an edge $e$ from $G$. Let $G_{1}$ be a subgraph of $G$, $G \backslash G_{1}$ is the subgraph obtained from $G$
by deleting the edges $G_{1}$.
For graph notation and concepts undefined here, readers are referred to \cite {A.B}.

\begin{lem}(\cite{D.C-1})\label{lem::2.1}
Let $G$ be a graph on $n$ vertices with maximum degree $\Delta(G)$. Then
\begin{equation*}
q_{1} (G) \leq 2\Delta(G).
\end{equation*}
Moreover, for a connected graph $G$, equality holds if and only if $G$ is regular.
\end{lem}

\begin{lem}(\cite{D.C-2})\label{lem::2.2}
The $Q$-index of any proper subgraph of a connected graph is smaller than the $Q$-index of the original graph.
\end{lem}

\begin{lem}(\cite{P.H})\label{lem::2.3}
Let $ K_{n}-e $ be a graph obtained from a complete graph $ K_{n} $ by deleting an edge $e$. Then
\begin{equation}\label{equ::1}
q_{1}(K_{n}-e)=\frac{3}{2}n-3+\frac{\sqrt{n^{2}+4n-12}}{2}.
\end{equation}
\end{lem}

\begin{lem}(\cite{H.L})\label{lem::2.4}
Let $ G^{e}_{n} $ be a graph obtained from a complete graph $ K_{n} $ by deleting an edge $e = uv$ and attaching edges $uu{'}$ and $vv{'}$ at $u$ and $v$, respectively. Then  $ q_{1}(G^{e}_{n}) $ is the largest real root of the equation
\begin{equation*}
x^{3} - (3n - 4) x^{2} + (n - 2)(2n - 1) x - 2(n-2)(n-3) = 0.
\end{equation*}
\end{lem}

\begin{lem}(\cite{R.M})\label{lem::2.5}
If $G$ is a connected graph, then
\begin{equation}\label{equ::2}
q_{1}(G)\leq \max\{d(u)+\frac{1}{d(u)}\sum_{v\in N(u)}d(v): u\in V(G)\},
\end{equation}
equality holds if and only if $G$ is either a regular graph or a semiregular bipartite graph.
\end{lem}

\begin{lem}(\cite{Y.H})\label{lem::2.6}
Let $G$ be a connected graph. Let $u, v$ be two distinct vertices of $G$ and $d(v)$ be the degree of vertex $v$. Suppose $v_{1}$,$v_{2}$,\dots,$v_{s}$($1 \leq s \leq d(v) $) are some vertices of $ N_{G}(v) \backslash N_{G}(u) $ and  $x=(x_{1},x_{2}, \dots,x_{n})^{T}$ is the Perron vector of $Q(G)$, where $x_{i}$ corresponds to the vertex $v_{i}$ ($1 \leq i \leq n$). Let $G^{'}$ be the  graph obtained from $G$ by deleting the edges $vv_{i}$ and adding the edges $uv_{i}$ ($1 \leq i \leq n$). If $ x_{u} \geq x_{v}$, then $ q_{1}(G) < q_{1}(G^{'}) $.
\end{lem}

\section{Proof of Theorem \ref{thm::1.5}}
Let $G^{*}$ be the extremal graph with the maximum $Q$-index among all $n$-vertex connected $K_{1,t}$-minor free graphs, where $t\geq3$ and $n \geq t+1$.
Denote by $q^{*}=q_{1}(G^{*})$ and $x=(x_{1},x_{2}, \dots,x_{n})^{T}$  the Perron vector of $Q(G^{*})$ corresponding to $q^{*}$. Furthermore, assume that $u^{*} \in V(G^{*})$ with $x_{u^{*}}=\max\limits_{u \in V(G^{*}) }x_{u}$, and let $A=N_{G^{*}}(u^{*})$, $B=V(G^{*}) \backslash (A \cup \{u^{*}\})$, $N^{2}(u^{*})=\{d_{G^{*}}(w,u^{*})=2 \mid w \in B\}$.

\begin{lem}\label{lem::3.1}
$ 2t - 3 < q^{*} \leq 2t - 2 $. Moreover, if $n\geq t+2$, then $q^{*}>2t - 2 - \frac{2}{t - 1}\ge 2t - 3$.
\end{lem}

\begin{proof}Since $G^{*} $ is $K_{1,t}$-minor free,
we have $ \Delta ( G^{*} ) \leq t-1 $.  By Lemma \ref{lem::2.1}, we have $q^{*} \leq 2t-2$.
Notice that $t\geq 3$,  from Lemma \ref{lem::2.3}  we have
\begin{equation*}
q_{1}(K_{t}-e)=\frac{3}{2}t-3+\frac{\sqrt{t^{2}+4t-12}}{2}\geq 2t - 3.
\end{equation*}
Note that $ S^{n-t} (K_{t}) $ contains $ K_{t}-e $ as a proper subgraph. By Lemma \ref{lem::2.2}, we have
$$ q^{*} \geq q_{1}(S^{n-t} (K_{t})) > q_{1}( K_{t}-e) \geq 2t-3.$$

Next we consider the case of $n \geq t + 2 $.
By Lemma \ref{lem::2.4}, $ q_{1}( G^{e}_{t}) $ is the largest real root of the equation $ g(x)=0 $, where
$$ g(x)= x^{3} - (3t - 4) x^{2} + (t - 2)(2t - 1) x - 2(t-2)(t-3).$$
For $t \geq 3$, we have
\begin{equation*}
\begin{aligned}
g(2t-2-\frac{2}{t-1}) =& (2t-2-\frac{2}{t-1})^{3} - (3t - 4) (2t-2-\frac{2}{t-1})^{2}\\
&+(t - 2)(2t - 1)(2t-2-\frac{2}{t-1}) - 2(t-2)(t-3)\\
=&(2t-2-\frac{2}{t-1})[(2t-2-\frac{2}{t-1})(-t-\frac{2}{t-1}+2)\\
&+(t - 2)(2t - 1)]- 2(t-2)(t-3)\\
=&(2t-2-\frac{2}{t-1})[t-\frac{2t}{t-1}+\frac{4}{(t-1)^{2}}-2]-2(t-2)(t-3)\\
=&-10+\frac{14}{t-1}+\frac{4}{(t-1)^{2}}-\frac{8}{(t-1)^{3}}\\<&0.
\end{aligned}
\end{equation*}
Thus
$q_{1}( G^{e}_{t})> 2t-2-\frac{2}{t-1} \geq 2t-3$.
Since $S^{n-t} (K_{t})$ contains $G^{e}_{t}$ as a proper subgraph, from Lemma \ref{lem::2.2}, we have
$$q^{*}\geq q_{1}(S^{n-t}(K_{t}))>q_{1}( G^{e}_{t})> 2t-2-\frac{2}{t-1} \geq 2t-3.$$
It completes the proof.	
\end{proof}

\begin{lem}\label{lem::3.2}
$ | A | = t - 1 $.
\end{lem}
\begin{proof}
Since $ G^{*}$ is $K_{1,t}$-minor free, we have $ |A | \leq \Delta ( G^{*} ) \leq t-1 $ . Assume for a contradiction that $ | A | \leq t-2 $. Then we have
\begin{equation*}
q^{*} \cdot x_{u^{*}} = | A | \cdot x_{u^{*}} + \sum\limits_{v \in A } x_{v} \leq 2 | A | \cdot  x_{u^{*}}.
\end{equation*}
So,  $ q^{*} \leq 2 | A | \leq 2t - 4$, which  contradicts with Lemma \ref{lem::3.1}.
\end{proof}

By the proof of Lemma \ref{lem::3.2}, we know that  $\Delta(G^{*})=t-1$. Note that $d_{A}(v) \leq t-2$ for any $v \in A $. In the following, we denote
$A_{0} = \{ v \in A \mid d_{A}(v) = t - 2 \}$ and $ A_{1} = \{ v \in A \mid d_{A}(v) \leq t - 3 \}=A\backslash A_0$.

\begin{lem}\label{lem::3.3}
\emph{(i)} $ d_{B} (v) = 0$ for any $ v \in A_{0} $.

\emph{(ii)} $ d_{B} (v) \leq 1$ for any $ v \in A_{1} $.

\emph{(iii)} $ d_{B} (w) \leq 2$ for any $ w \in B $. Particularly, $ d_{B} (w) \leq 1$ for any $ w \in N^{2} (u^{*}) $.
\end{lem}

\begin{proof}
(i) Clearly, $d_{B}(v)=0$ for any $v \in A_{0}$ since $\Delta(G^{*})=t-1$.
	
(ii) Suppose that there exists a vertex $v \in A_{1}$ such that $d_{B}(v) \geq 2$. Then $G^{*}$ contains a double star with the non-pendant edge $u^{*}v$ and $|A|-1+d_{B}(v)$ leaves. Notice that $|A|-1+d_{B}(v) \geq t$, which implies that $G^{*}$ contains a $K_{1,t}$-minor, a contradiction. Thus (ii) holds.
	
(iii) Note that $G^{*}$ is connected.  For any vertex $w \in B$, there exists a shortest path $P$ from $w$ to one of vertices in $A$. Clearly, $d_{B \cap V(P)}(w) \leq 1$. Moreover, $d_{B \cap V(P)}(w)=0$ if $w \in N^{2}(u^{*})$. Let $v \in A$ be the other endpoint of $P$. Then $G^{*}$ contains a tree consisting of a path $P$ and an edge $v u^{*}$ with $|A|-1+d_{B \backslash V(P)}(w)$ leaves. Since $G^{*}$ is $K_{1, t}$-minor free, we know that $|A|-1+d_{B \backslash V(P)}(w) \leq t-1$ and   $d_{B}(w) = d_{B \cap V(P)}(w)+d_{B \backslash V(P)}(w)$. It follows that
$ d_{B} (w) \leq 2$ for any $ w \in B $. Particularly, $ d_{B} (w) \leq 1$ if $ w \in N^{2} (u^{*})$. Thus (iii) holds.
\end{proof}

\begin{lem}\label{lem::3.4}
If $n=t+1$, suppose that $B=\{w\}$, then we have

\emph{(i)} $N_{G^*}(w)=A_{1}$;

\emph{(ii)} $d_{A}(v)=t-3$ for any $v\in A_{1}$.
\end{lem}

\begin{proof}
(i) If there exists a vertex $v \in A_{1}$ with $v w \notin E(G^{*})$, then
$\max \{d_{G^{*}}(v), d_{G^{*}}(w)\} \leq t-2,$
and thus $\Delta(G^{*}+v w)=\Delta(G^{*})=t-1$, which implies $G^{*}+vw$ contains no $K_{1, t}$-minor.
By Lemma \ref{lem::2.2}, we have $q_{1}(G^{*}+v w) > q^{*}$,  a contradiction. Thus (i) holds.

(ii) Recall that $A_{1}=\{v \in A \mid d_{A}(v) \leq t-3\}$. If there exists a vertex $v_{0} \in A_{1}$ such that $d_{A}(v_{0}) \leq t-4$. Take a vertex $v_{1} \in \overline{N}_{A}(v_{0})$, then $d_{A}(v_{1}) \leq t-3$, and by (i) $v_{1} \in A_{1}=N_{G^{*}}(w)$. In fact, $d_{A}(v_{1})=t-3$, otherwise adding an edge to $v_{0}$ and $v_{1}$ which leads to a $K_{1,t}$-minor free graph with larger $Q$-index, a contradiction. Thus we have
\begin{equation*}\begin{array}{ll}
q^{*} \cdot x_{u^{*}} = d(u^{*}) \cdot x_{u^{*}} + \sum\limits_{v \in A } x_{v} = (t-1) \cdot x_{u^{*}} + \sum\limits_{v \in A } x_{v},\\
q^{*} \cdot x_{v_{1}} = d(v_{1}) \cdot x_{v_{1}} + \sum\limits_{v \in N(v_{1}) } x_{v} = (t-1) \cdot x_{v_{1}} +x_{u^{*}} + x_{w} +\sum\limits_{v \in A \backslash \{v_{0},v_{1}\}} x_{v},
\end{array}
\end{equation*}
and then we obtain
\begin{equation*}
x_{w} =(q^{*}-t+2)  (x_{v_{1}} -x_{u^{*}} ) +  x_{v_{0}}.
\end{equation*}
By Lemma \ref{lem::3.1}, $q^{*} -t +2 > 0$.
Thus $x_{w} \leq x_{v_{0}}$ due to $x_{u^{*}} \geq x_{v_{1}}$.
Let $G^{'}=G^{*}-\{w v_{1}\}+\{v_{0} v_{1}\}$. Obviously, $G^{'}$ is still a connected graph since $v_{0} \in A_{1}=N_{G^{*}}(w)$, and $G^{'}$ is a $K_{1,t}$-minor free graph. By Lemma \ref{lem::2.6}, we have $q_{1}(G^{'}) > q^{*}$, which contradicts the maximality of $q^{*}$. It follows (ii).
\end{proof}

We firstly give the proof of Theorem \ref{thm::1.5} for $n=t+1$.
\renewcommand\proofname{\bf Proof of Theorem \ref{thm::1.5} for $n=t+1$.}
\begin{proof}
From Lemma \ref{lem::3.4}, $d_{A}(v)=t-3$ for any $v\in A_{1}$ and $A_{0}\cup\{w, u^*\}$ dominates $A_1$, we get that $G^{*}[A_{1}]$ is isomorphic to $K_{|A_{1}|}$ by deleting a perfect matching $M$. Thus $|A_{1}|$ is even.

By Lemma \ref{lem::2.1}, we know that $q_{1}(G) \leq 2\Delta(G)$ for any connected graph $G$, with equality if and only if $G$ is $\Delta$-regular. Noting that $G^{*}$ is $\Delta(G^{*})$-regular if and only if $A_{1}=A$ by Lemma \ref{lem::3.2}. Therefore, if $n=t+1$ is even, then $G^{*} \cong K_{n} \backslash \frac{n}{2} K_{1,1}$, where the $\frac{n}{2}$ independent edges deleting from $K_{n}$ are the union of the edge $u^{*} w$ with a perfect matching $M$ in $A$.

If $n=t+1$ is odd, then $|A|=t-1$ is also odd. Since $|A_{1}|$ is even, we have $|A_{1}|<|A|$. By symmetry and the Perron-Frobenius theorem, all vertices of $G^{*}$ in $A_{0}\cup \{u^{*}\}$ or $A_{1}$  have the same eigenvector components, which are denoted by $x_{u^{*}}$ and $x_{1}$, respectively. Thus we have
\begin{equation*}\left\{\begin{array}{ll}
q^{*} \cdot x_{w}=|A_{1}| \cdot x_{w} + |A_{1}| \cdot x_{1},\\
q^{*} \cdot x_{u^{*}}= (t-1)\cdot x_{u^{*}} + |A_{1}|\cdot x_{1}+|A_{0}| \cdot x_{u^{*}},\\
q^{*} \cdot x_{1}=(t-1)\cdot x_{1} + (|A_{1}|-2) \cdot x_{1}+(|A_{0}|+1) \cdot x_{u^{*}}+x_{w}.
\end{array}\right.
\end{equation*}
Noting that $|A_{0}|+|A_{1}|=|A|=t-1$. By above equations, we get
\begin{equation}\label{equ::3}
q^{*3}+(5-3t-|A_{1}|) q^{*2}+(3|A_{1}|t-8t-5|A_{1}|+2t^{2}+6)q^{*} + c=0.
\end{equation}
where $c=-2|A_{1}|^{2}-(2t^{2}-10t+8)|A_{1}|$. Since $|A_{1}|$ is even and $n=t+1$ is odd, we have $2 \leq |A_{1}| \leq t-2$ and $t \geq 4$, and thus $2t^{2}-10t+8 \geq 0$. Therefore the function
$f(x) = -2x^{2}-(2t^{2}-10t+8)x$
is decreasing for $2 \leq x \leq t-2$, which implies that $\min\limits_{2 \leq x \leq t-2} f(x)=f(t-2)$.
From (\ref{equ::3}), we have
$$c=-q^{*3}-(5-3t-|A_{1}|) q^{*2}-(3|A_{1}|t-8t-5|A_{1}|+2t^{2}+6)q^{*}.$$
Moreover, let
\begin{equation}\label{equ::4}
c(q)=-q^{3}-(5-3t-|A_{1}|) q^{2}-(3|A_{1}|t-8t-5|A_{1}|+2t^{2}+6)q.
\end{equation}
Then
$$c^{'}(q)=-3q^{2}-2(5-3t-|A_{1}|) q-(3|A_{1}|t-8t-5|A_{1}|+2t^{2}+6).$$
Since
\begin{equation*}
\begin{aligned}
r=& [-2(5-3t-|A_{1}|)]^{2}-4\times3\times(3|A_{1}|t-8t-5|A_{1}|+2t^{2}+6)\\
=&28-24t+20|A_{1}|+12t^{2}-12|A_{1}|t+4|A_{1}|^{2}\\
\geq&28-24t+20\times2+12t^{2}-12(t-2)t+4\times2^{2}\\
>&0,
\end{aligned}
\end{equation*}
by solving the equation $c^{'}(q_0)=0$, we have
$$q_0=\left\lbrace \frac{2(5-3t-|A_{1}|)+\sqrt{r}}{-6}, \frac{2(5-3t-|A_{1}|)-\sqrt{r}}{-6}\right\rbrace. $$
Now, we see that
\begin{equation*}
\begin{aligned}
&\max q_0-(2t-3)\\
=& \max \left\lbrace \frac{2(5-3t-|A_{1}|)+\sqrt{r}}{-6}, \frac{2(5-3t-|A_{1}|)-\sqrt{r}}{-6}\right\rbrace -(2t-3)\\
=&\frac{2(3t+|A_{1}|-5)+\sqrt{28-24t+20|A_{1}|+12t^{2}-12|A_{1}|t+4|A_{1}|^{2}}}{6} -(2t-3) \\
\leq&\frac{2(3t+(t-2)-5)+\sqrt{28-24t+20(t-2)+12t^{2}-12\times2t+4(t-2)^{2}}}{6} -(2t-3)\\
=&\frac{(8t-14)+\sqrt{(4t-5)^{2}-4t-21}}{6} -(2t-3)\\
<&\frac{(8t-14)+(4t-5)}{6} -(2t-3)\\
<&0.
\end{aligned}
\end{equation*}
Thus $c(q)$ in (\ref{equ::4}) is decreasing for $q > 2t-3$. Since $q^{*}$ is the maximum $Q$-index and $q^{*} > 2t-3$ by Lemma \ref{lem::3.1}, we get that $c$ in   (\ref{equ::3}) must attain its minimum value. Therefore, $|A_{1}|=t-2$, and  $|A_{0}|=1$, and so  $G^{*} \cong K_{n} \backslash \frac{n-1}{2} K_{1,1}$ for odd $n$.
It completes the proof.
\end{proof}

In order to give the proof of Theorem \ref{thm::1.5} for $n\geq t+2$,
we let $z^{*}\in V(G^{*})$ such that
\begin{equation*}
d(z^{*})+\frac{1}{d(z^{*})}\sum_{v\in N(z^{*})}d(v)=\max\{d(u)+\frac{1}{d(u)}\sum_{v\in N(u)}d(v): u\in V(G)\}.
\end{equation*}
From Lemma \ref{lem::2.5}, we have
\begin{equation}\label{equ::5}
q^{*} \leq d(z^{*})+\frac{1}{d(z^{*})}\sum_{v\in N(z^{*})}d(v).
\end{equation}
\renewcommand\proofname{\bf Proof}
\begin{lem}\label{lem::3.5}
$d(z^{*})=t-1$.
\end{lem}

\begin{proof}
By Lemma \ref{lem::3.2}, we know that $ d(z^{*}) \leq \Delta(G^{*}) = t - 1 $. On the other hand, from Lemma \ref{lem::3.1} and    (\ref{equ::5}), we have
\begin{equation*}
\begin{aligned}
2t-3 < q^{*} \leq& d(z^{*})+\frac{1}{d(z^{*})}\sum_{v\in N(z^{*})}d(v) \leq d(z^{*})+ \Delta(G^{*}) = d(z^{*}) +t-1,
\end{aligned}
\end{equation*}
which implies that $d(z^{*}) > t-2$, and thus $d(z^{*})=t-1$.
\end{proof}

\begin{lem}\label{lem-3.5'} If $n\geq t +2$, then $ |A_{1}| \geq |N^{2} (u^{*}) | \geq 2$.\end{lem}
\begin{proof}
We first show that $|N^{2}(u^{*})| \geq 2$.
Note that $G^{*}$ is connected. Since $|A|=t-1$ and   $n \geq t+2$,
we have $N^{2}(u^{*}) \neq \emptyset$ and $|B|\geq 2$.
Suppose by the contrary that $|N^{2}(u^{*})| = 1$.
Let $N^{2}(u^{*})=\{w_0\}$ and $w_{1}\in N_B(w_0)$, and thus $d(w_{0}) \geq 2$.
Taking a vertex $v_{0} \in N_{A}(w_{0})$.
If $\overline{N}_{A}(w_{0}) \subseteq N_{G^{*}}(v_{0})$,
then
$$|N_{G^{*}}(v_{0}) \cup N_{G^{*}}(w_{0})|=|A \cup\{u^{*}, w_{0}, w_{1}\}|=t+2,$$
which implies   $G^{*}$ contains a $K_{1,t}$-minor, a contradiction.
Thus there exists a vertex, say $v_{1}$, in
$\overline{N}_{A}(w_{0}) \cap \overline{N}_{G^{*}}(v_{0})$. Thus from Lemma \ref{lem::3.3} (ii), $d_{G^{*}}(v_{1}) \leq t-2$. By Lemma \ref{lem::3.3} (iii), $G^{*}[B]$ is a pendant path $P$ of order $n-t$ with the endpoint $w_{0}$, say $P=$ $w_{0} w_{1} \cdots w_{n-t-1}$. Suppose to that $d(w_{0})=2$,
then $N_{A}(w_{0})=\{v_{0}\}$.
Let $G^{'}=G^{*}+\{v_{1} w_{n-t+1}\}$.
Clearly, $\Delta(G^{'})=\Delta(G^{*})=t-1$ and $G^{'}$ is a subgraph of $S^{n-t}(K_{t})$, and so $G^{'}$ is $K_{1, t}$-minor free.
By Lemma \ref{lem::2.2}, we have $q_{1}(G^{'}) > q^{*}$, a contradiction.
So $d(w_{0})\geq 3$,  and thus $t\geq 4$.
Suppose that $t=4$, then $|A|=3$. Let $\{v_{2}\}=A \backslash \{v_{0},v_{1}\}$.
Then $N_A(w_0)=\{v_0,v_2\}$, and $v_{2}$ is adjacent to at most one of $\{v_{0},v_{1}\}$. Now let $G^{'}=G^{*}+\{v_{1} w_{n-t+1}\}$. Clearly, $G^{'}$ is $K_{1,t}$-minor free. By Lemma \ref{lem::2.2}, we have $q_{1}(G^{'}) > q^{*}$, also a contradiction. Therefore, $t\geq5$.
Now we consider the following three cases to lead a contradiction, respectively.

{\flushleft\bf Case 1. If $z^{*} \in A_{0}\cup \{u^{*}\}$.}

Notice that $z^{*}v_{1}\in E(G^*)$. We have $d_{G^{*}}(v_{1}) = t-2$ since  otherwise $d_{G^{*}}(v_{1}) \leq t-3$. By  Lemma \ref{lem::3.1} and (\ref{equ::5}), we have
\begin{equation*}
\begin{aligned}
2t - 2 - \frac{2}{t - 1} < q^{*} \leq& d(z^{*})+\frac{1}{d(z^{*})}\sum_{v\in N(z^{*})}d(v)\\\leq& d(z^{*}) + \frac{(d(z^{*})-1)\Delta(G^{*})+d_{G^{*}}(v_{1})}{d(z^{*})}\\\leq&t-1 + \frac{(t-2)(t-1)+t-3}{t-1}\\=&2t- 2 - \frac{2}{t - 1},
\end{aligned}
\end{equation*}
a contradiction. Thus $\overline{N}_{A_{1}}(v_{1})=\{v_{0}\}$.
Since $d(w_{0})\geq 3$, we may assume $v_{2} \in N_{A}(w_{0}) \backslash \{v_{0}\}$.
If $\overline{N}_{A}(w_{0}) \subseteq N_{G^{*}}(v_{2})$, then
$$|N_{G^{*}}(v_{2}) \cup N_{G^{*}}(w_{0})|=|A \cup\{u^{*}, w_{0}, w_{1}\}|=t+2,$$
which implies that $G^{*}$ contains a $K_{1,t}$-minor, a contradiction.
Thus there exists a vertex, say $v_{3} \in \overline{N}_{A}(w_{0}) \cap \overline{N}_{G^{*}}(v_{2})$, where $v_{3} \neq v_{1}$. Clearly, $d_{G^{*}}(v_{3})\leq t-2$. Notice that $z^{*}$ is adjacent to $v_{3}$. By (\ref{equ::5}) and Lemma \ref{lem::3.1}, we get
\begin{equation*}
\begin{aligned}
2t - 2 - \frac{2}{t - 1} < q^{*} \leq& d(z^{*})+\frac{1}{d(z^{*})}\sum_{v\in N(z^{*})}d(v)\\\leq& d(z^{*}) + \frac{(d(z^{*})-2)\Delta(G^{*})+d_{G^{*}}(v_{3})+d_{G^{*}}(v_{1})}{d(z^{*})}\\\leq&t-1 + \frac{(t-3)(t-1)+2(t-2)}{t-1}\\=&2t- 2 - \frac{2}{t - 1},
\end{aligned}
\end{equation*}
a contradiction.
{\flushleft\bf Case 2. If $z^{*} \in A_{1}$.}

Since $d_{G^{*}}(z^{*})=t-1$, we have $z^{*} \in N_{A}(w_{0})$. Without loss of generality, we may assume $z^{*} = v_{0}$, and so $d_{G^{*}}(v_{0})=t-1$ and $\overline{N}_{A_{1}}(v_{0})=\{v_{1}\}$. Furthermore,
$ d(w_{0})\geq t-2$. Otherwise, if $ d(w_{0}) \leq t-3$, by (\ref{equ::5}) and Lemma \ref{lem::3.1}, we have
\begin{equation*}
\begin{aligned}
2t - 2 - \frac{2}{t - 1} < q^{*} \leq& d(z^{*})+\frac{1}{d(z^{*})}\sum_{v\in N(z^{*})}d(v)\\\leq& d(z^{*}) + \frac{(d(z^{*})-1)\Delta(G^{*})+d_{G^{*}}(w_{0})}{d(z^{*})}\\\leq&t-1 + \frac{(t-2)(t-1)+t-3}{t-1}\\=&2t- 2 - \frac{2}{t - 1},
\end{aligned}
\end{equation*}
which is a contradiction. Now let $G^{'} = G - z^{*}w_{0} + z^{*}v_{1}$.
Since $d(w_{0})\geq t-2$ and $G^{*}$ is connected, we have $G^{'}$ is connected.
Note that $\Delta(G^{'})=\Delta(G^{*})=t-1$, and so $G^{'}$ contains no $K_{1,t}$-minor. On the other hand,
from the eigen-equations of $ Q(G^{*}) $,
\begin{equation*}
\begin{array}{ll}
q^{*}\cdot x_{u^{*}} = d(u^{*}) \cdot x_{u^{*}} + \sum\limits_{v \in A } x_{v} = (t-1) x_{u^{*}} +x_{z^{*}}+x_{v_1} + \sum\limits_{v \in A\backslash \{z^{*},v_{1}\} } x_{v},\\
q^{*}\cdot  x_{z^{*}} = d(z^{*}) \cdot x_{z^{*}} + \sum\limits_{v \in N(z^{*}) } x_{v} = (t-1)  x_{z^{*}} +x_{u^{*}} + x_{w_{0}} +\sum\limits_{v \in A\backslash \{z^{*},v_{1}\}} x_{v}.
\end{array}
\end{equation*}
Thus
$(q^{*}-t +2)(x_{u^{*}}-x_{z^{*}}) =x_{v_{1}}-x_{w_{0}}.$
Since $q^{*} -t +2 > 0$ by Lemma \ref{lem::3.1} and $x_{u^{*}} \geq x_{z^{*}}$,
we have $x_{v_{1}} \geq x_{w_{0}}$. By Lemma \ref{lem::2.6}, we get $ q_{1}(G^{'}) > q^{*}$, a contradiction.

{\flushleft\bf Case 3. If $z^{*} \in B$.}

Recall that $d_{G^{*}}(z^{*})=t-1$.
Since  $d(w_{i})\leq 2$ for $i=1,2,\cdots, n-t-1$,
we have $z^{*} = w_{0}$. By (\ref{equ::5}) and Lemma \ref{lem::3.1}, we get
\begin{equation*}
\begin{aligned}
2t - 2 - \frac{2}{t - 1} < q^{*} \leq& d(z^{*})+\frac{1}{d(z^{*})}\sum_{v\in N(z^{*})}d(v)\\\leq& d(z^{*}) + \frac{(d(z^{*})-1)\Delta(G^{*})+d_{G^{*}}(w_{1})}{d(z^{*})}\\\leq&t-1 + \frac{(t-2)(t-1)+2}{t-1}\\=&2t- 2 - \frac{t-3}{t - 1},
\end{aligned}
\end{equation*}
a contradiction.
Thus $|N^{2}(u^{*})| \geq 2$.
By Lemma \ref{lem::3.3} (ii), we have $|A_{1}| \geq|N^{2}(u^{*})|\ge 2$.
\end{proof}

\begin{lem}\label{lem::3.6}
$ d_{B} (v) = 1$ for any $ v \in A_{1} $.
\end{lem}

\begin{proof}
Suppose to the contrary that there exists a vertex $ v_{0} \in A_{1}$ such that $d_{B} (v_{0}) = 0 $. Clearly, $ d_{G^{*}} (v_{0}) \leq t-2 $ and $d(v)\leq t-2$ for any $v\in B$.
Recall that $ d(z^{*}) = t-1$. So $z^{*} \notin B \cup \{v_{0}\} $.
Thus $z^{*} \in (A_{0}\cup \{u^{*}\})\cup (A_{1}\backslash \{v_{0}\})$. Next we will distinguish  two cases to lead a contradiction.

{\flushleft\bf Case 1. $z^{*} \in A_{0}\cup \{u^{*}\}$.}

Firstly, we have $d(v_{0})=t-2$ since otherwise  if $d(v_{0})\leq t-3$, by  Lemma \ref{lem::3.1} and (\ref{equ::5}), we have
\begin{equation*}
\begin{aligned}
2t - 2 - \frac{2}{t - 1} < q^{*} \leq& d(z^{*})+\frac{1}{d(z^{*})}\sum_{v\in N(z^{*})}d(v)\\\leq& d(z^{*}) + \frac{(d(z^{*})-1)\Delta(G^{*})+d_{G^{*}}(v_{0})}{d(z^{*})}\\\leq&t-1 + \frac{(t-2)(t-1)+t-3}{t-1}\\=&2t- 2 - \frac{2}{t - 1},
\end{aligned}
\end{equation*}
a contradiction. Thus $|\overline{N}_{A}(v_{0})\cap A_{1}| = 1$.
Let   $\{v_{1}\}=\overline{N}_{A}(v_{0})\cap A_{1}$. Then $d(v_{1})=t-1$. Otherwise, if $d(v_{1})\leq t-2$, by Lemma \ref{lem::3.1} and  (\ref{equ::5}), we have
\begin{equation*}
\begin{aligned}
2t - 2 - \frac{2}{t - 1} < q^{*} \leq& d(z^{*})+\frac{1}{d(z^{*})}\sum_{v\in N(z^{*})}d(v)\\\leq& d(z^{*}) + \frac{(d(z^{*})-2)\Delta(G^{*})+d_{G^{*}}(v_{1})+d_{G^{*}}(v_{0})}{d(z^{*})}\\\leq&t-1 + \frac{(t-3)(t-1)+2(t-2)}{t-1}\\=&2t- 2 - \frac{2}{t - 1},
\end{aligned}
\end{equation*}
a contradiction. Thus from Lemmas \ref{lem::3.2} and  \ref{lem::3.3}(ii), we have $d_{B}(v_{1})=1$ and $v_{1}$ is adjacent to all vertices in $A\backslash \{v_{0},v_{1}\}$. Let $N_{B}(v_{1})=\{w_{1}\}$. Since $|N^{2} (u^{*}) | \geq 2$ by Lemma \ref{lem-3.5'}, there exists a vertex $w_{2} \in N^{2} (u^{*})\backslash \{w_{1}\}$.
Choosing  a vertex, say $v_{2}$, in $N_{A}(w_{2})$.
Clearly,  $v_{2}$ is adjacent to $v_{0}$ and $v_{1}$.
Therefore, we have
$$|N_{G^{*}}(v_{1})\cup N_{G^{*}}(v_{2})|=|A\cup \{u^{*},w_{1},w_{2}\}|=t+2,$$
which implies $G^*$ has a $K_{1,t}$-minor, a contradiction.
{\flushleft\bf Case 2. $z^{*} \in A_{1}\backslash \{v_{0}\}$.}

By Lemmas \ref{lem::3.2} and \ref{lem::3.3}(ii) and the fact $d(z^{*})=t-1$, we have $d_{B}(z^{*})=1$. Let $N_{B}(z^{*})=\{w_{1}\}$.
By Lemmas \ref{lem::3.2}, \ref{lem::3.3} and \ref{lem-3.5'}, we have $d(w_{1})\leq t-2$.
Moreover, we claim that $d(w_{1})=t-2$. Otherwise, if $d(w_{1})\leq t-3$, by (\ref{equ::5}) and Lemma \ref{lem::3.1}, we have
\begin{equation*}
\begin{aligned}
2t - 2 - \frac{2}{t - 1} < q^{*} \leq& d(z^{*})+\frac{1}{d(z^{*})}\sum_{v\in N(z^{*})}d(v)\\\leq& d(z^{*}) + \frac{(d(z^{*})-1)\Delta(G^{*})+d_{G^{*}}(w_{1})}{d(z^{*})}\\\leq&t-1 + \frac{(t-2)(t-1)+t-3}{t-1}\\=&2t- 2 - \frac{2}{t - 1},
\end{aligned}
\end{equation*}
a contradiction. Thus $|N^{2}(u^{*})|\leq2$ by Lemma \ref{lem::3.3}
and the fact $|A|=t-1$. By Lemma \ref{lem-3.5'}, we get $|N^{2}(u^{*})|=2$ . Let $w_{2}=N^{2}(u^{*})\backslash \{w_{1}\}$. Then $|N_{A}(w_{2})|=1$. Moreover, we have $z^{*} \in \overline{N}_{A}(v_{0})\cap A_{1}$ since otherwise  if $z^{*}$ is adjacent to $v_{0}$. By Lemma \ref{lem::3.1} and  (\ref{equ::5}), we obtain
\begin{equation*}
\begin{aligned}
2t - 2 - \frac{2}{t - 1} < q^{*} \leq& d(z^{*})+\frac{1}{d(z^{*})}\sum_{v\in N(z^{*})}d(v)\\\leq& d(z^{*}) + \frac{(d(z^{*})-2)\Delta(G^{*})+d_{G^{*}}(w_{1})+d_{G^{*}}(v_{0})}{d(z^{*})}\\\leq&t-1 + \frac{(t-3)(t-1)+2(t-2)}{t-1}\\=&2t- 2 - \frac{2}{t - 1},
\end{aligned}
\end{equation*}
a contradiction.
Now let $G^{'} = G - z^{*}w_{1} + z^{*}v_{0} $. Clearly, $\Delta(G^{'})=\Delta(G^{*})=t-1$ and
$$\left\{\begin{array}{ll}
N_{G^{'}}(z^{*}) = (N_{G^{*}}(z^{*}) \backslash \{w_{1}\}) \cup \{v_{0}\},\\
N_{G^{'}}(v_{0}) = N_{G^{*}}(v_{0}) \cup \{z^{*}\},\\
N_{G^{'}}(w_{1}) = N_{G^{*}}(w_{1}) \backslash \{z^{*}\}.
\end{array}\right.
$$
If $ d(w_{1})=t-2= 2$, then  $N_{A}(w_{1})=\{z^{*}\}$, which implies that $G^{'}$ is a subgraph of $S^{n-t}(K_{t})$, and so $G^{'}$ is $K_{1, t}$-minor free.
If $ d(w_{1})=t-2 \geq 3$, then $t\geq 5$.
Since  $G^{*}$ is connected, we have $G^{'}$ is connected.
Taking  $u_{0} \in N_{G^{*}}(v_{0})$ and $u_{1} \in N_{G^{*}}(z^{*}) \backslash \{w_{1}\}$.
Clearly, $u_{0},u_{1} \in A \cup \{u^{*}\}$.  By Lemma \ref{lem::3.3}(i) and (ii) we have $|N_{B}(u_{0})|, |N_{B}(u_{1})|\leq 1$.   Thus we obtain
$$
\left\{\begin{array}{ll}
|N_{G^{'}}(z^{*})\cup N_{G^{'}}(u_{1})|=|A\cup \{u^{*}\} \cup N_{B}(u_{1})|\leq t+1,\\
|N_{G^{'}}(v_{0})\cup N_{G^{'}}(u_{0})|\leq|A\cup \{u^{*}\} \cup N_{B}(u_{0})|\leq t+1,\\
|N_{G^{'}}(v_{0})\cup N_{G^{'}}(z^{*})|=|A\cup \{u^{*}\} |= t.
\end{array}\right.
$$
Therefore $G^{'}$ is also $K_{1,t}$-minor free.
By the eigen-equation of $Q(G^{*})$, we have
\begin{equation*}
\begin{array}{ll}
q^{*} \cdot x_{u^{*}} = d(u^{*}) \cdot x_{u^{*}} + \sum\limits_{v \in A } x_{v} = (t-1) \cdot x_{u^{*}} +x_{z^{*}}+x_{v_0} + \sum\limits_{v \in A\backslash \{z^{*},v_{0}\} } x_{v},\\
q^{*} \cdot x_{z^{*}} = d(z^{*}) \cdot x_{z^{*}} + \sum\limits_{v \in N(z^{*}) } x_{v} = (t-1) \cdot x_{z^{*}} +x_{u^{*}} + x_{w_{1}} +\sum\limits_{v \in A\backslash \{z^{*},v_{0}\}} x_{v}.
\end{array}
\end{equation*}
Then we get
$(q^{*} -t +2)   (x_{u^{*}}-x_{z^{*}}) =x_{v_{0}} -x_{w_{1}}.$
Notice that  $q^{*} -t +2 > 0$ from Lemma \ref{lem::3.1}.
We have $x_{v_{0}} \geq x_{w_{1}}$, and thus $ q_{1}(G^{'}) > q^{*}$ from Lemma \ref{lem::2.6}, which contradicts the maximality of $q^{*}$.

It completes the proof.
\end{proof}

Set $ A_{10} = \{ v \in A_{1} \mid d_{A}(v) = t - 3 \} $ and $ A_{11}=A_{1} \backslash A_{10} = \{ v \in A_{1} \mid d_{A}(v) \leq t - 4 \}$.

\begin{lem}\label{lem::3.7}
For $ | A_{1}|  \geq 3 $, $ G^{*}[A_{10}]$ is a clique. Moreover, $|\overline{N}_{A_{11}}(v)|=1$ for each vertex $v \in A_{10}$.
\end{lem}

\begin{proof}
Suppose to the contrary that there exists two non-adjacent vertices $v_{1}, v_{2} \in A_{10}$. Since $d_{A}(v_{1})=d_{A}(v_{2})=t-3$, we get that both $v_{1}$ and $v_{2}$ dominate $A \backslash\{v_{1}, v_{2}\}$. Since $|A_{1}| \geq 3$, there exists a vertex $v_{3} \in A_{1} \backslash\{v_{1}, v_{2}\}$.
By Lemma \ref{lem::3.6}, there exists a unique vertex $w_{i} \in N_{B}(v_{i})$ for $i \in\{1,2,3\}$.

If $w_{3} \neq w_{1}$, then $|N_{G^{*}}(v_{1}) \cup N_{G^{*}}(v_{3})| = |A \cup\{u^{*}, w_{1}, w_{3}\}|=t+2$.
If $w_{3} \neq w_{2}$, then we similarly have $|N_{G^{*}}(v_{2}) \cup N_{G^{*}}(v_{3})| = t+2$.	
If $w_{1}=w_{2}=w_{3}$. Since $|N^{2}(u^{*})| \geq 2$ by Lemma \ref{lem-3.5'}, there exists a vertex $w_{4} \in N^{2}(u^{*}) \backslash\{w_{1}\}$ and a vertex $v_{4} \in N_{A}(w_{4})$. It follows that
$$|N_{G^{*}}(v_{1}) \cup N_{G^{*}}(v_{4})| = |A \cup\{u^{*}, w_{1}, w_{4}\}|=t+2.$$
Above three cases, we always find a $K_{1, t}$-minor,   a contradiction.
Thus
$G^{*}[A_{10}]$ is a clique. In addition, by the definition of $A_{10}$, $|\overline{N}_{A}(v)|=1$ for any $v \in A_{10}$.
\end{proof}

\begin{lem}\label{lem::3.8}
 $|A_{11}| \leq 2 $.
\end{lem}

\begin{proof}
Suppose by the contrary that $| A_{11}| \geq 3 $.
Note that   $d_{G^{*}}(v) \leq t-2$ for any $v \in A_{11}$ and
 $ d_{G^{*}} (w) \leq 2$ for any $ w \in B \backslash N^{2}(u^{*})$ from Lemma \ref{lem::3.3}.
By Lemma \ref{lem::3.5} we know that $d(z^{*})=t-1$, and thus $z^{*} \notin A_{11} \cup (B \backslash N^{2}(u^{*})) $.
That is, $z^{*} \in  (A_{0} \cup \{u^{*}\}) \cup A_{10} \cup N^{2}(u^{*})$.
Next we will consider   three cases to leads a contradiction, respectively.
{\flushleft\bf Case 1. $z^{*} \in A_{0} \cup \{u^{*}\}$.}

Note that all vertices in $A_{0} \cup \{u^{*}\}$ are adjacent to the vertices in $A_{11}$. From (\ref{equ::5}) and Lemma \ref{lem::3.1}, we have
\begin{equation*}
\begin{aligned}
2t - 2 - \frac{2}{t - 1} < q^{*} \leq& d(z^{*})+\frac{1}{d(z^{*})}\sum_{v\in N(z^{*})}d(v)\\
\leq& d(z^{*}) + \frac{(d(z^{*})-|A_{11}|)\Delta(G^{*})+|A_{11}|(t-2)}{d(z^{*})}\\
=&2t-2 - \frac{|A_{11}|}{t-1}\\
\leq & 2t- 2 - \frac{3}{t - 1},
\end{aligned}
\end{equation*}
which is a contradiction.

{\flushleft\bf Case 2. $z^{*} \in A_{10}$. }

Since $| A_{11}| \geq 3 $, we have $| A_{1}| \geq 3 $. By Lemma \ref{lem::3.7}, $ A_{10} $ induces a clique. By the definition of $A_{10}$,   $z^{*}$ has just a non-neighbor, say $z_0$ in $A_{11}$. Thus, $| A_{11}| \geq 3 $ implies that $z^{*}$ is adjacent to at least two vertices in $A_{11}$. By  (\ref{equ::5}) and Lemma \ref{lem::3.1}, we obtain
\begin{equation*}
\begin{aligned}
2t - 2 - \frac{2}{t - 1} < q^{*} \leq& d(z^{*})+\frac{1}{d(z^{*})}\sum_{v\in N(z^{*})}d(v)\\
\leq& d(z^{*}) + \frac{(d(z^{*})-(| A_{11}|-1))\Delta(G^{*})+\sum_{v\in A_{11} \backslash \{z_0\}}d(v)}{d(z^{*})}\\
=&2t-2 - \frac{(| A_{11}|-1)(t-1)-\sum_{v\in A_{11} \backslash \{z_0\}}d(v)}{t-1}\\ \leq& 2t-2 - \frac{(| A_{11}|-1)[(t-1)-\max \left\lbrace d(v) \mid v \in A_{11}\right\rbrace]}{t-1}\\
\leq& 2t-2 - \frac{|A_{11}|-1}{t-1}\\
\leq&2t- 2 - \frac{2}{t - 1},
\end{aligned}
\end{equation*}
 a contradiction.

{\flushleft\bf Case 3. $z^{*} \in N^{2}(u^{*})$.}

From Lemmas \ref{lem::3.3}(iii) and   \ref{lem-3.5'}, we have $ d_{B} (w) \leq 1$ for any $ w \in N^{2} (u^{*}) $ and $|N^{2} (u^{*}) | \geq 2$. Notice that $d(z^{*})=t-1$ and $|A|=t-1$. Then $z^{*}$ is adjacent to $t-2$ vertices in $A$.
  Note that $| A_{11}| \geq 3 $.  we get that $z^{*}$ is adjacent to at least two vertices $z_1, z_2 \in A_{11}$. By  Lemma \ref{lem::3.1} and (\ref{equ::5}), we have
\begin{equation*}
\begin{aligned}
2t - 2 - \frac{2}{t - 1} < q^{*} \leq& d(z^{*})+\frac{1}{d(z^{*})}\sum_{v\in N(z^{*})}d(v)\\
\leq& d(z^{*}) + \frac{(d(z^{*})-2)\Delta(G^{*})+d(z_1)+d(z_2)}{d(z^{*})}\\
\leq&t-1 + \frac{(t-3)(t-1)+2(t-2)}{t-1}\\
=&2t- 2 - \frac{2}{t - 1},
\end{aligned}
\end{equation*}
a contradiction.
\end{proof}

\begin{lem}\label{lem::3.9}
Let $ v_{i}^{o} \in A_{10} $, $ w_{i} \in N_{B} (v_{i}^{o}) $ and $ \overline{N}_{A}(v_{i}^{o}) = \{v_{i}\} $, then $ x_{w_{i}} \leq x_{v_{i}}$.
\end{lem}

\begin{proof}
From the eigen-equations for $ Q(G^{*}) $, we see that
\begin{equation*}\begin{array}{ll}
q^{*} \cdot x_{u^{*}} = d(u^{*}) \cdot x_{u^{*}} + \sum\limits_{v \in A } x_{v} = (t-1) \cdot x_{u^{*}} + \sum\limits_{v \in A } x_{v},\\
q^{*} \cdot x_{v_{i}^{o}} = d(v_{i}^{o}) \cdot x_{v_{i}^{o}} + \sum\limits_{v \in N(v_{i}^{o}) } x_{v} = (t-1) \cdot x_{v_{i}^{o}} +x_{u^{*}} + x_{w_{i}} +\sum\limits_{v \in A \backslash \{v_{i},v_{i}^{o} \} } x_{v}.
\end{array}
\end{equation*}
It follows that
$x_{w_{i}} =(q^{*}-t+2) \cdot (x_{v_{i}^{o}} -x_{u^{*}} ) +  x_{v_{i}}$.
Since $q^{*} -t +2 > 0$ by Lemma \ref{lem::3.1} and $x_{v_{i}^{o}} \leq x_{u^{*}}$, we get $ x_{w_{i}} \leq  x_{v_{i}}$.
\end{proof}

\renewcommand\proofname{\bf Proof of Theorem \ref{thm::1.5} for $n\geq t+2$}
\begin{proof}
If $|A_{1}|=2$, then $|N^{2}(u^{*})|=2$ and $d_{A}(w)=1$ for each $w \in N^{2}(u^{*})$. Thus $G^{*}[A \cup\{u^{*}\}] \cong K_{t}-e$, where the unique non-edge lies in $A_{1}$. By Lemma \ref{lem::3.3}(iii), $G^{*}[B]$ is the union of at most two paths. 
Suppose $G^{*}[B]$ is the union of exactly two paths,  adding an edge to two pendant vertices in $B$, which leads to a $K_{1,t}$-minor free graph with larger $Q$-index, a contradiction. Therefore, $G^{*}[A_{1} \cup B]$ is a path with both endpoints in $A_{1}$, and thus $G^{*} \cong S^{n-t}(K_{t})$, as desired.

It remains to consider the case for $|A_{1}| \geq 3$ by Lemma \ref{lem-3.5'}.
Next we will  show that  $|A_{1}| \geq 3$  is impossible.
Noting that $|A_{11}| \leq 2$ from Lemma \ref{lem::3.8}. We will distinguish   two cases to lead a contradiction, respectively.

{\flushleft\bf Case 1.} \textbf{$|A_{11}|=2$.}

Let $A_{11}=\{v_{1}, v_{2}\}$. By the definition of $A_{11}$, $d_{G^{*}}(v_1), d_{G^{*}}(v_2) \leq t-2$ and each $v_{i}$ has at least one non-neighbor, say $v_{i}^{o}$ in $A_{10}$ for $i \in\{1,2\}$. Moreover, $d_{A}(v_{1}^{o})=d_{A}(v_{2}^{o})=t-3$ and $v_{1}^{o} \neq v_{2}^{o}$.
Thus $|A|=t-1\ge|A_{1}|\geq 4$, and so $ t\geq 5$.
By Lemma \ref{lem::3.3} (iii), $ d_{G^{*}} (w) \leq 2$ for each $w \in B \backslash N^{2}(u^{*})$.
Recall that $d(z^{*})=t-1$. Thus we have
\begin{equation}\label{z*}
z^{*}\in A_{0} \cup \{u^{*}\}\cup A_{10}\cup N^{2}(u^{*}).
\end{equation}

Suppose to that $z^{*} \in A_{0} \cup \{u^{*}\}$.
Notice that all vertices in $A_{0} \cup \{u^{*}\}$ are adjacent to the vertices in $A_{11}$.
By  Lemma \ref{lem::3.1} and (\ref{equ::5}), we get
\begin{equation*}
\begin{aligned}
2t - 2 - \frac{2}{t - 1} < q^{*} \leq& d(z^{*})+\frac{1}{d(z^{*})}\sum_{v\in N(z^{*})}d(v)\\\leq& d(z^{*}) + \frac{(d(z^{*})-|A_{11}|)\Delta(G^{*})+|A_{11}|(t-2)}{d(z^{*})}\\=&2t-2- \frac{|A_{11}|}{t-1}\\=&2t- 2 - \frac{2}{t - 1},
\end{aligned}
\end{equation*}
which is a contradiction.

Suppose to that $z^{*} \in N^{2}(u^{*})$.
Notice that $ d_{B} (w) \leq 1$ for any $ w \in N^{2} (u^{*}) $ and $|N^{2} (u^{*}) | \geq 2$.
Since $d(z^{*})=t-1$,
we get that $z^{*}$ is adjacent to $t-2$ vertices in $A$ and a vertex in $B$.
Thus, $|N^{2} (u^{*}) |=2$.
Let $N_{B}(z^{*})=\{w\}$.
If $w \in B \backslash N^{2}(u^{*})$, then $d_{G^{*}}(w)\leq 2$ by Lemma \ref{lem::3.3}(iii).
If $w \in N^{2}(u^{*})\backslash \{z^{*}\}$,
then $\{w\}= N^{2}(u^{*})\backslash \{z^{*}\}$, so $w$ is adjacent to $z^{*}$ and a vertex in $A$, and so $d_{G^{*}}(w) = 2$. Thus
\begin{equation}\label{d(w)}d_{G^{*}}(w)\leq 2.\end{equation}
Since $|A|=t-1$, $d(z^{*})=t-1$ and  $| A_{11}| = 2 $,
$z^{*}$ is adjacent to at least one vertex, say $v_{1}$ in $A_{11}$, and $d(v_1)\leq t-2$.
By Lemma \ref{lem::3.1}, (\ref{equ::5}) and (\ref{d(w)}), we have
\begin{equation*}
\begin{aligned}
2t - 2 - \frac{2}{t - 1} < q^{*} \leq& d(z^{*})+\frac{1}{d(z^{*})}\sum_{v\in N(z^{*})}d(v)\\
\leq& d(z^{*}) + \frac{(d(z^{*})-2)\Delta(G^{*})+d(v_{1})+d_{G^{*}}(w)}{d(z^{*})}\\
\leq& t-1 + \frac{(t-3)(t-1)+(t-2)+2}{t-1}\\
=&2t- 2 - \frac{t-2}{t - 1},
\end{aligned}
\end{equation*}
a contradiction. Thus by  (\ref{z*}) we have  $z^{*} \in A_{10}$.
Without loss of generality, we may assume that $z^{*}=v_{1}^{o}\in A_{10}$. By Lemma \ref{lem::3.6}, there exists a unique vertex $w_{1} \in N_{B}(v_{1}^{o})$.
Clearly, $d(w_{1}) \leq t-1$.
Suppose that $d(w_{1})\leq t-2$.
Recall that $v_2\in A_{11}$, and thus $d(v_{2}) \leq t-2$.
By Lemma \ref{lem::3.1} and (\ref{equ::5}), we have
\begin{equation*}
\begin{aligned}
2t - 2 - \frac{2}{t - 1} < q^{*} \leq& d(v_{1}^{o})+\frac{1}{d(v_{1}^{o})}\sum_{v\in N(v_{1}^{o})}d(v)\\\leq&  d(v_{1}^{o}) + \frac{(d(v_{1}^{o})-2)\Delta(G^{*})+d_{G^{*}}(v_{2})+d_{G^{*}}(w_{1})}{d(v_{1}^{o})}\\\leq&t-1 + \frac{(t-3)(t-1)+2(t-2)}{t-1}\\=&2t- 2 - \frac{2}{t - 1},
\end{aligned}
\end{equation*}
a contradiction, Thus $d(w_{1}) = t-1$.
Suppose that $d(v_{2}) \leq t-3$. By Lemma \ref{lem::3.1} and  (\ref{equ::5}), we get
\begin{equation*}
\begin{aligned}
2t - 2 - \frac{2}{t - 1} < q^{*} \leq& d(v_{1}^{o})+\frac{1}{d(v_{1}^{o})}\sum_{v\in N(v_{1}^{o})}d(v)\\\leq& d(v_{1}^{o}) + \frac{(d(v_{1}^{o})-1)\Delta(G^{*})+d_{G^{*}}(v_{2})}{d(v_{1}^{o})}\\\leq& t-1 + \frac{(t-2)(t-1)+(t-3)}{t-1}\\=&2t- 2 - \frac{2}{t - 1},
\end{aligned}
\end{equation*}
a contradiction, Thus $d(v_{2})=t-2$.

Recall that $ d_{B} (w) \leq 1$ for any $ w \in N^{2} (u^{*}) $, $|N^{2} (u^{*}) | \geq 2$, $|A| = t-1$ and $d(w_{1})=t-1$, where $N_{B}(v_{1}^{o})=\{w_{1}\}$.
Also, $ d_{B} (v) = 1$ for any $ v \in A_{1} $.
We get that $ |N^{2} (u^{*}) | = 2$ and $d_{B}(w_{1}) = 1$.
Let $N_B(w_1)=w_{1}^{o}$ and $ \{w_{2} \}=N^{2} (u^{*}) \backslash \{w_{1}\}$.
Since $G^{*}$ is a $K_{1,t}$-minor free graph, we see that all vertices in $A \backslash \{v_{1}\}$ are adjacent to
$w_{1}$ and $v_{1}$ is adjacent to $w_{2}$. By the definition of $A_{10}$ and Lemma \ref{lem::3.6},   $d(v_{2}^{o}) = t-1$. Thus, $v_{2}^{o}v_{1}\in E(G^*)$.
Now
$|N_{G^{*}}(v_{2}^{o}) \cup N_{G^{*}}(w_{1})| = |A \cup \{ u^{*} , w_{1} , w_{1}^{o} \}|= t+2 $, which indicates that   $G^{*}$ has a $K_{1, t}$-minor,   a contradiction.

{\flushleft\bf Case 2. $|A_{11}|\leq 1$.}

By the definition of $A_{0}$ and Lemma \ref{lem::3.7},
each vertex $v \in A_{10}$ has a non-neighbor $v_{0}$ in $A_{11}$.
Thus $A_{11} \neq \emptyset$ and so $|A_{11}| = 1$,
$A_{11}=\{v_{0}\}$ and $e(\{v_{0}\}, A_{10})=0$. Note that $|A_{1}| \geq 3$, then $|A_{10}|=|A_{1}|-|A_{11}| \geq 2$. By Lemma \ref{lem::3.6}, there exists a unique vertex $w_{i} \in N_{B}(v_{i})$ for each $v_{i} \in A_{1}$, possibly, $w_{i}=w_{j}$ for some $i \neq j$.
By Lemma \ref{lem-3.5'}, we know that $|N^{2}(u^{*})| \geq 2$.
Moreover, since $G^{*}$ is connected,
we get that each component of $G^{*}[B]$ is a path of order at least one by Lemma \ref{lem::3.3}(iii).
Now we distinguish the following two subcases.

{\flushleft\bf Subcase 2.1. $|N^{2}(u^{*})| = 2$.}
Then there exists a vertex $v_{k} \in A_{10}$ such that $\{w_{k}\}= N^{2}(u^{*})\backslash \{w_{0}\}$.
Let $G^{'}=G^{*}-\{v_{i} w_{i} \}+\{v_{i} v_{0}\}$ for all $v_{i} \in A_{10} \backslash \{v_{k}\}$.
By Lemma \ref{lem::3.9}, we know  $x_{w_{i}} \leq x_{v_{0}}$ for all $v_{i} \in A_{10} \backslash \{v_{k}\}$, and by
Lemma \ref{lem::2.6}, we have $q^{*} < q_{1}(G^{'})$.
Notice that $G^{'}[A \cup\{u^{*}\}] \cong K_{t}-e$,
where $v_{0}v_{k}$ is the unique non-edge. Moreover, $G^{'}[B]$ is still the disjoint union of paths, and the edges between $A$ and $B$ just are $v_{0}w_{0}$ and $v_{k}w_{k}$. Therefore, $G^{'}$ is a subgraph of $S^{n-t}(K_{t})$. It follows that $q_{1}(S^{n-t}(K_{t})) \geq q_{1}(G^{'})>q^{*}$, a contradiction.

{\flushleft\bf Subcase 2.2. $|N^{2}(u^{*})| \geq 3$.}
By Lemma \ref{lem::3.3}(iii),
there exists a vertex $w_{k} \in N^{2}(u^{*})\backslash \{w_{0}\}$
such that $w_{k}$ and $w_{0}$ belong to two distinct paths of $G^{*}[B]$.
Denote by  $v_{k} \in N_{A}(w_{k})$.
Let $G^{'}=G^{*}-\{v_{i} w_{i} \}+\{v_{i} v_{0}\}$ for all $v_{i} \in A_{10} \backslash \{v_{k}\}$, $w_{i} \in N^{2}(u^{*})\backslash \{w_{0},w_k\}$.
Since  $x_{w_{i}} \leq x_{v_{0}}$ from Lemma \ref{lem::3.9},
we have $ q^{*} < q_{1}(G^{'})$  by Lemma \ref{lem::2.6}.
Clearly, $G^{'}[A \cup\{u^{*}\}] \cong K_{t}-e$, where $v_{0} v_{k}$ is the unique non-edge.
Moreover, $G^{'}[B]$ is still the disjoint union of some paths, and the edges between $A$ and $B$ just are $v_{0} w_{0}$ and $v_{k} w_{k}$.
Therefore, $G^{'}$ is a subgraph of $S^{n-t}(K_{t})$. It follows that $q_{1}(S^{n-t}(K_{t})) \geq q_{1}(G^{'})>q^{*}$, a contradiction.

It completes the proof.
\end{proof}



\end{document}